\documentclass[12pt]{amsart}

\usepackage[T1]{fontenc}  
\usepackage[utf8]{inputenc}  
\usepackage{lmodern}  
\usepackage{amsmath, amssymb, amsfonts, amscd}  

\usepackage{graphicx}  
\usepackage{xcolor}  
\usepackage{subcaption}  
\usepackage[all]{xy}  

\usepackage{geometry}
\newtheorem{theorem}{Theorem}[section]
\newtheorem{lemma}[theorem]{Lemma}
\theoremstyle{definition}
\newtheorem{definition}[theorem]{Definition}
\newtheorem{example}[theorem]{Example}
\newtheorem{proposition}[theorem]{Proposition}
\newtheorem{corollary}[theorem]{Corollary}

\newtheorem{conjecture}[theorem]{Conjecture}

\theoremstyle{remark}
\newtheorem{remark}[theorem]{Remark}

\usepackage{hyperref}  
\hypersetup{
    colorlinks=true,
    linkcolor=blue,
    citecolor=blue,
    urlcolor=blue,
    pdfauthor={Your Name},
    pdftitle={Your Paper Title}
}

\begin{document}
\title[Infinite Geometrically Distinct Unknots] {Geometric Constraints in Link Isotopy}
\author{JOS\'E AYALA}
\address{Universidad de Tarapacá, Iquique, Chile}
              \email{jayalhoff@gmail.com}


\baselineskip=20 true pt
\baselineskip=1.10\normalbaselineskip
\maketitle 

\begin{abstract} 
\baselineskip=20 true pt
\maketitle \baselineskip=1.10\normalbaselineskip
We prove the existence of families of distinct isotopy classes of physical unknots through the key concept of parametrised thickness. These unknots have prescribed length, tube thickness, a uniform bound on curvature, and cannot be disentangled into a thickened round circle by an isotopy that preserves these constraints throughout. In particular, we establish the existence of \emph{gordian unknots}: embedded tubes that are topologically trivial but geometrically locked, confirming a long-standing conjecture. These arise within the space \( \mathcal{U}_1 \) of thin unknots in \( \mathbb{R}^3 \), and persist across a stratified family \( \{ \mathcal{U}_\tau \}_{\tau \in [0,2]} \), where \( \tau \) denotes the tube diameter, or thickness. The constraints on curvature and self-distance fragment the isotopy class of the unknot into infinitely many disconnected components, revealing a stratified structure governed by geometric thresholds. This unveils a rich hierarchy of geometric entanglement within topologically trivial configurations.
\end{abstract}

\section{Motivation and Context} 

The study of thick, or physical knots, is based on an idealised model where the rope is perfectly flexible, sectionally incompressible, frictionless, and satisfies a normalised bound on curvature of at most one and thickness two. In this theory, curvature and thickness, are saturated: they jointly maximise the allowable tube diameter.


In classical knot theory, an unknot is a loop that can be isotopically deformed into a round circle. However, when constraints on length, thickness, and curvature are considered, it was believed that there exist unknots that resist isotopies attempting to untangle them into a thickened round circle. These configurations are known as gordian unknots (their name comes from the Greek legend of the gordian knot \cite{guardian}). In this note, we establish the existence of such objects.

\begin{figure}[htbp]
   \centering
   \includegraphics[width=.8\linewidth]{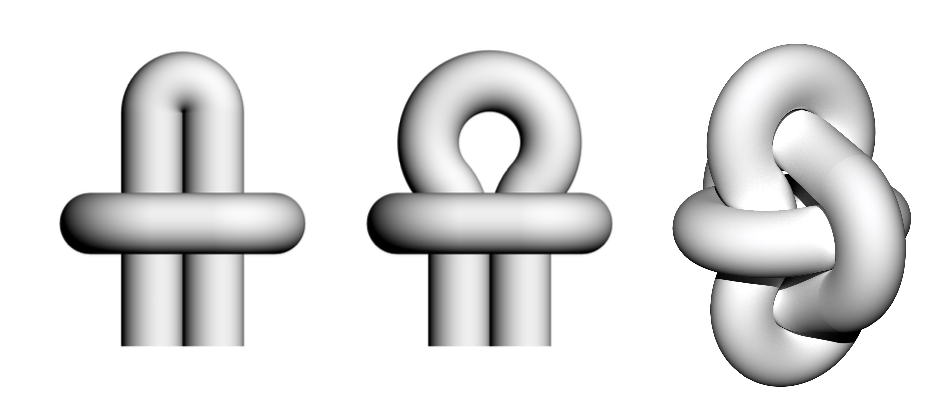}
   \caption{Left: This represents the standard approach to geometric knots. Both the cord and the horizontal ring maintain a uniform thickness of one, but their curvatures are bounded differently, two for the cord and one for the stadium ring. Under these conditions, the ring can slide freely along the cord without obstruction. Center: The ring to get stuck. In this case both the cord and the ring have curvature bounded above by one and thickness one. Right: Thin Borromean rings with curvature bounded by one and thickness one. This choice of thickness parameter is suitable for self-linking.}
   \label{fig:fitting}
\end{figure}

A key insight arises from the distinction between thin and thick knots, leading to families of physical knots parametrised by their tube thickness. In \cite{paperi}, we constructed examples of $2$-component gordian unlinks by detecting isotopy classes that differ from those in classical knot theory. We obtained a $1$-parameter family of minimal length unlinks with thickness varying in the interval $[1,2)$, and observed that thinner tubes give rise to rope geometries distinct from the thick case, where the thickness equals two and the tube saturates its curvature-limited normal injectivity radius, eliminating the geometric “slack” that enables thin knot constructions, see figure~\ref{fig:fitting}.

More generally, varying the thickness parameter $\tau \in [0,2)$ reveals a continuum of geometric behaviours. At one extreme, thickness zero corresponds to curves constrained only by curvature (i.e., $\kappa$-constrained embeddings). At the other, thickness two recovers the theory of thick knots, where the knot fills its normal injectivity radius. In this case, the geometry becomes maximally self-avoiding, the tube is locally constrained by its own thickness, and certain motions or isotopies become geometrically forbidden. When the radius of curvature exceeds the tube radius, the normal vectors fail to converge within the tubular neighbourhood. This prevents self-intersection and allows the configuration to remain smooth and embedded, even under tight geometric constraints.

The thin model extends thick knot theory by allowing the tube diameter to drop below the curvature maximum. While embeddedness and curvature remain enforced, the absence of a bulky tube permits greater geometric freedom. As $\tau \to 0$, curvature becomes the dominant constraint, giving rise to a rich space of admissible configurations. This regime reveals a stratified structure of entanglements shaped by fine-scale obstructions. In \cite{paperi}, we constructed an example showing that a gordian entanglement can occur even in the zero-thickness limit, showing that lockedness arises from curvature alone, independent of excluded volume effects.

Thin knots exhibit a subtle form of scale-dependent rigidity. As the tube diameter decreases, global flexibility increases, yet local isotopies are increasingly challenged by fine-scale geometric features that begin to dominate and obstruct smooth deformations. In addition, the spaces of thin unknots, are stratified, as thickness vanishes, the space fragments into an increasingly intricate hierarchy of isotopy classes governed by fine-scale geometric obstructions.

In 1994, an earlier candidate for gordian unknot, proposed by Freedman, He, and Wang \cite{freedman}, was later disentangled by the work of Kusner and Sullivan \cite{kusner_sullivan}. Subsequently in 2001, Pierański, Przybyl, and Stasiak proposed a thick unknot that could not be disentangled using the SONO (Shrink On No Overlap) algorithm \cite{pieranski1}. More recently in 2015, Coward and Hass \cite{cowardhass} proved the existence of a gordian split link, consisting of two unlinked thick knots that cannot be disentangled without violating geometric constraints. In 2025 Ayala and Hass provided the first examples of thick gordian unlinks, \cite{ayalahass}. 

The distinction between thin and thick knots is not merely theoretical, it manifests in physical systems. It is the relationship between thickness and minimum bend radius that ultimately determines the shape of the rope. For instance, take any cord, bend it into a U-turn, and push it through itself, the result is a kind of turnbuckle shape at the pushed end, see figure~\ref{fig:fitting}. Now repeat the experiment with a thicker cord: the resulting turnbuckle is likely thicker and more compact, reflecting the reduced flexibility of the core. This illustrates how increasing thickness imposes tighter geometric constraints. However, no ordinary cord or wire can form perfectly sharp bends as permitted by the idealised thick model, where singularities are intrinsic part of the mathematical model.

The results here presented indicate obstructions to extending Hatcher's proof of the Smale conjecture \cite{hatcher 1} within the framework of geometric knot theory, separating this with classical knot theory at a fundamental level.

\section{Cores of Thin Tubes}\label{spaces}

The uniform bound on absolute curvature satisfied by the class of curves considered as cores of our physical knots gives rise to obstructions when attempting certain continuous deformations. This constraint restrict the flexibility of the curves, preventing transitions between specific configurations.

\begin{definition} \label{cbc} An embedded arc-length parameterised curve $\gamma: [0,s]\rightarrow {\mathbb R}^3$ is called {\it $\kappa$-constrained} if:
\begin{itemize}
\item $\gamma$ is $C^1$ and piecewise $C^2$
\item $||\gamma''(t)||\leq \kappa$, for all $t\in [0,s]$ when defined, $\kappa>0$ a constant.
\end{itemize}
 If $\gamma(0)=\gamma(s)$ and $\gamma'(0)=\gamma'(s)$ then $\gamma$ is called a knot. If $\gamma(0)\neq \gamma(s)$ then $\gamma$ is called an arc. 
\end{definition}

 The $\kappa$-constrained curves have absolute curvature bounded above almost everywhere by a positive constant. By normalising this constant $1/\kappa = 1$, the minimum admissible radius of curvature is set to 1. The space of arcs connecting fixed points $x,y\in {\mathbb R}^3$ is denoted by $\Sigma(x,y)$ and it is considered with the $C^1$ metric. In \cite{paperi} we proved that when $0<||x-y||<2$ the space $\Sigma(x,y)$ has two connected components one containing exclusively embedded arcs, see Theorems 3.7 and 3.9 therein.
 
 Next we assert that a $1$-constrained arc within a 3-ball of radius $1$ cannot intersect the boundary of the ball at a single, isolated point, see Lemma 2.2 in \cite{paperi}.

\begin{lemma}\label{3dntp} 
Suppose a $1$-constrained arc $\gamma:[0,s]\to \mathbb R^3$ is defined in a radius $1$ 3-ball $B$. Then, $\gamma([0,s]) \subset \partial B$ or $\gamma((0,s)) \cap \partial B=\emptyset$. 
\end{lemma}

Next result highlights a geometric obstruction to continuous deformations of $1$-constrained curves. This obstruction depends exclusively on curvature and the distance between fixed endpoints of an arc. When the distance between the endpoints of an arc $\gamma$ is less than $2$, the line segment connecting $\gamma(0)$ to $\gamma(s)$ and another arc connecting the same endpoints, having a point sufficiently above the line segment, belong to distinct homotopy classes in the space of $1$-constrained arcs in $\mathbb{R}^3$ connecting these fixed endpoints, see Lemma 2.5 in \cite{paperi}. 

\begin{lemma}\label{r1}  
(Geometric obstruction). A $\kappa$-constrained arc $\gamma: [0,s]\to{\mathcal C}$ such that:
$${\mathcal C}=\{(x,y,z)\in{\mathbb R}^3\,|\, x^2+y^2<1,\,z\geq 0 \}$$
is an open cylinder, cannot satisfy both:
\begin{itemize}
\item $\gamma(0),\gamma(s)$ are points on the $xy$-plane.
\item If $S$ is a radius $1$ sphere with centre on the negative $z$-axis, and $\gamma(0),\gamma(s)\in S$. Then, some point in the image of $\gamma$ lies above $S$.
\end{itemize}
In addition, if $\gamma$ satisfy {\rm (2)} then its diameter is at least 2.
\end{lemma}

\begin{remark}\hfill
\begin{enumerate}
\item The key idea in the proof of Lemma \ref{r1} is to construct an arc that makes a long U-turn while remaining entirely within the open cylinder ${\mathcal C}$. Then, through compactness, Lemma \ref{3dntp} is applied at the maximum height of the arc to conclude that such an arc cannot exist.
\item The sharp U-turn obstruction, derived from the curvature bound in Lemma~\ref{r1}, is independent of thickness. In Corollary~4.3 of \cite{paperi}, we proved that, under a curvature constraint alone, there exist locked unlinks of thickness zero.
\end{enumerate}
\end{remark}

\begin{definition} \label{shortlong}Let $B$ be a radius $1$ 3-ball centred at the $z$-axis. A short arc has its endpoints on $\partial B$ and it is defined entirely on $\partial B$, or it is defined in $int(B)$ except at its endpoints. A long arc has its endpoints on $\partial B$ and has a point above $S$. Therefore:
\begin{itemize}
\item a short arc satisfies (1) in Lemma \ref{r1}
\item a long arc satisfies (2) in Lemma \ref{r1}.
\end{itemize}
By Lemma \ref{r1} these are mutually exclusive
\end{definition}

\section{Thin and Thick Links: Parametrised Thickness and Stratification}

This section introduces the concept of parametrised thickness for $1$-constrained knots and links, with a focus on the nested structure of the associated spaces. 

A \(1\)-constrained knot or link is an embedding of one or more closed curves in \(\mathbb{R}^3\). We adopt the notion of thickness from \cite{simon1}, using the embedded tube diameter, which better captures its geometric meaning; see also \cite{diao1, cantarella 1}. The thickness is defined as:
\[
\mathrm{Thi}(\gamma) = \min \left\{ 2, R_2(\gamma) \right\}
\]
where \( R_2(\gamma) \) denotes the minimal distance between all pairs of \emph{doubly critical points}; that is, pairs \( p, q \in \gamma \) such that the chord \( \overline{pq} \) is orthogonal to the tangent vectors at both endpoints.

This formulation reflects the interplay between curvature and self-avoidance. The curvature bound limits the maximum allowable tube diameter to two, while the doubly critical distance constrains local proximity. 

\begin{definition}
A \emph{thin knot} or link is a $1$-constrained embedding of fixed length with fixed tube thickness in the interval \( [0,2) \). Two such knots are said to be \emph{thin isotopic} if they are connected by a $1$-constrained isotopy that preserves both length and tube thickness throughout. Knots or links that are isotopic but not thin isotopic are called \emph{gordian}.
\end{definition}

\begin{definition}
For each \( \tau \in [0,2] \), we define \( \mathcal{U}_\tau \) as the space of all $1$-constrained unknots whose thickness satisfies \( \mathrm{Thi}(\gamma) \geq \tau \). That is,
\[
\mathcal{U}_\tau = \left\{ \gamma \colon S^1 \to \mathbb{R}^3 \,\middle|\, \mathrm{Thi}(\gamma) \geq \tau \text{ and } \gamma \text{ is an unknot} \right\}.
\]
\end{definition}

This definition treats thickness as a geometric constraint imposed on the admissible configurations, rather than a fixed property of a knot type. In particular, a given knot may admit multiple realisations with varying thickness and thus belong to several \( \mathcal{U}_\tau \) simultaneously.

\subsection*{Monotonicity and stratification.} The family \( \{ \mathcal{U}_\tau \}_{\tau \in [0,2]} \) is nested by geometric admissibility. That is,
\[
\tau > \tau' \quad \text{implies} \quad \mathcal{U}_\tau \subset \mathcal{U}_{\tau'},
\]
which expresses the monotonicity property: any configuration admissible at thickness \( \tau \) remains admissible for all smaller values.

\begin{example}
Let \( \tau = 2 \) and \( \tau' = 1 \). Consider a round circle \( \gamma \) of radius 1. Its curvature is constant and equal to 1, satisfying the fixed bound \( \kappa = 1 \). The shortest distance between doubly critical pairs occurs at antipodal points and equals 2. Hence,
\[ 
\mathrm{Thi}(\gamma) = \min\left\{ 2, R_2(\gamma) \right\} = \min\{2, 2\} = 2 \geq \tau.
\]
Therefore, \( \gamma \in \mathcal{U}_2 \subset \mathcal{U}_1 \), illustrating the monotonicity property.
\end{example}

 The monotonicity endows the family \( \{ \mathcal{U}_\tau \}_{\tau \in [0,2]} \) with a stratified structure, indexed by thickness. As the parameter \( \tau \) decreases, the constraint on self-distance weakens, and the corresponding space \( \mathcal{U}_\tau \) becomes strictly larger. The nested inclusions reflect a filtration by geometric admissibility:
\[
\mathcal{U}_2 \subset \mathcal{U}_\tau \subset \mathcal{U}_{\tau'} \subset \cdots \subset \mathcal{U}_\epsilon,
\quad \text{for all } 2 > \tau > \tau' > \cdots > \epsilon \geq 0.
\]
We define the space of \emph{thin unknots} as the union of all spaces with strictly sub-maximal thickness:
\[
\mathcal{T} = \bigcup_{0 \leq \tau < 2} \mathcal{U}_\tau.
\]
Note the limit space \( \mathcal{U}_2 \) is the space of \emph{thick unknots}. The case \( \tau = 0 \) is included to capture the limiting case of $1$-constrained unknots without any enforced self-distance constraint. As shown in \cite{paperi}, such zero-thickness embeddings can exhibit nontrivial isotopy obstructions within the $1$-constrained regime, and serve as limiting cases of thin gordian structures. Each \( \mathcal{U}_\tau \) is endowed with the \( C^1 \) metric on the space of immersions or the Hausdorff metric on tubular neighbourhoods. 

The existence of thin gordian unknots reveals that geometric constraints, particularly curvature and thickness, can obstruct isotopies even in topologically trivial settings. This highlights a fundamental distinction between the topological and geometric classification of knots. From a geometric standpoint, such configurations reveal that the stratified space \( \{ \mathcal{U}_\tau \} \) encodes more than topological information: the ability to deform a knot is tightly controlled by the parameter \( \tau \), with certain entanglements persisting even in the absence of excluded volume. These phenomena underscore the need for a refined theory of isotopy classes under geometric constraints, and are relevant in applied contexts involving entangled filaments, such as DNA organisation and soft robotics.

\begin{remark}
The thin case offers a relaxed framework for studying the ropelength problem. As the thickness parameter \( \tau \) increases, the condition
\[
\mathrm{Thi}(\gamma) = \min\left\{2, R_2(\gamma)\right\} \geq \tau
\]
admits fewer and fewer configurations, since all doubly critical pairs must be at least distance \( \tau \) apart. At \( \tau = 2 \), both curvature and self-distance constraints are maximally saturated, corresponding to the thick rope case, where flexibility is sharply limited and isotopies are highly constrained. In contrast, thinner tubes allow greater geometric freedom. These relaxed models preserve essential features of the ropelength functional while broadening the admissible configuration space for analysis and deformation.
\end{remark}


Inspired by the classical work of Dubins and Sussmann \cite{dubins 1, sussman, papera} on curvature-constrained paths, we conjecture that length-minimising thin physical unknots in \( \mathbb{R}^3 \) admit a similarly structured, yet more intricate, decomposition.

To date, the only exact ropelength known in the entire theory of thick knots is that of the round circle. After decades of efforts not a single nontrivial knot has an exact ropelength value. We propose a canonical decomposition that could lead to the first exact solutions beyond the unknot.

\begin{conjecture} 
Every length-minimizing unknot in $\mathcal{U}_\tau$ for $\tau\in (0,2)$ admits a decomposition into finitely many elementary geometric components:
\begin{itemize}
    \item Circular arcs of unit radius (saturating the curvature bound),
    \item Straight segments (minimizing length locally),
    \item Sussmann helices (mediating between these extremes),
\end{itemize}
with the number of elementary components bounded by $O(1/\tau)$. 
\end{conjecture}

\begin{conjecture}[Stratified Obstruction Theory]
For each \( \tau \in (0,2] \), the space \( \mathcal{U}_\tau \) of \( 1 \)-constrained physical unknots contains countably infinite isotopy classes. As \( \tau \to 0 \), geometric obstructions, such as bottlenecks, cone-angle collapses, and curvature traps, emerge at finer scales and proliferate, fragmenting the space into more rigid and distinguishable components.

These isotopy classes organise into a stratified space \( \mathcal{K} = \bigcup_{\tau \leq 2} \mathcal{U}_\tau \) in the sense of Mather: each stratum corresponds to configurations with a fixed obstruction profile, and strata of higher complexity accumulate onto simpler ones as thickness decreases. This structure reflects a scale-sensitive geometry, where isotopy flexibility decays through a hierarchy of singularities, revealing a deeply layered phase of physical entanglement.
\end{conjecture}


\begin{remark}[Scale-Sensitive Isotopy]
A scale-sensitive isotopy is one whose admissibility depends on the geometric resolution set by the tube diameter \( \tau \). For large \( \tau \), the thick tube restricts both bending and proximity between strands. As \( \tau \to 0 \), volume constraints weaken, but curvature continues to obstruct motion. As shown in Lemma~\ref{r1}, even at zero thickness, a long arc cannot pass through a planar aperture of radius 2. Thus, as \( \tau \to 0\), the configuration space fragments under fine-scale geometric obstructions, and isotopy classes may split into geometrically distinct components, locked not by volume, but by curvature alone.
\end{remark}

\section{Infinite Geometrically Distinct Unknots}

Diao, Ernst, and van Rensburg \cite{diao1998} conducted laboratory experiments on open knots (thick ropes confined between walls) to approximate energy-minimising configurations, showing consistency of theoretical results with numerical simulations. Later on Pierański, Przybył, and Stasiak studied tight open knots \cite{pieranski3} using an analogous setup to \cite{diao1998} but focusing on the numerical analysis of curvature, torsion, and the symmetric behaviour observed by these for some small knots.

\begin{theorem}\label{main1}
There exists a gordian unknot.
\end{theorem}

\begin{proof}\noindent\textbf{Open knots.}
Let \( \gamma:[0,s] \to \mathbb{R}^3 \) be a \( 1 \)-constrained open overhand knotted core of a tube of thickness two, satisfying:
\begin{enumerate}
    \item \( \gamma(0) = (0,0,0) \), \( \gamma(s) = (0,0,h) \),
       \item the tangents \( \gamma'(t) \) are parallel to the \( z \)-axis at \( t = 0 \) and \( t = s \).
    \item \( \gamma \) and its tube are entirely contained between the planes \( z = 0 \) and \( z = h \), see figure~\ref{fig:render9}.
\end{enumerate}
\smallskip
Apply the SONO algorithm \cite{pieranski2, pieranski3} until reaching a nearly minimal ropelength realisation.

\smallskip
\noindent\textbf{The unknotted double overhand.} We properly embed two parallel tubes of thickness one each inside the nearly minimal thickness two open overhand and treat each as an open curve with curvature bound one and thickness one. The SONO algorithm is reapplied to the double tube until a nearly optimal configuration is reached. Then at the open ends, planar Dubins paths are attached as caps (both of type \textsc{ccc}, since their tangents are parallel, opposite oriented and distant apart one \cite{dubins 1, papera}) forming an embedded thin unknot $K_0$ with a double overhand core at the middle and two long arcs one in each side, see figure~\ref{fig:render9}.

{ \begin{figure} 
 \begin{center}
\includegraphics[width=1\textwidth,angle=0]{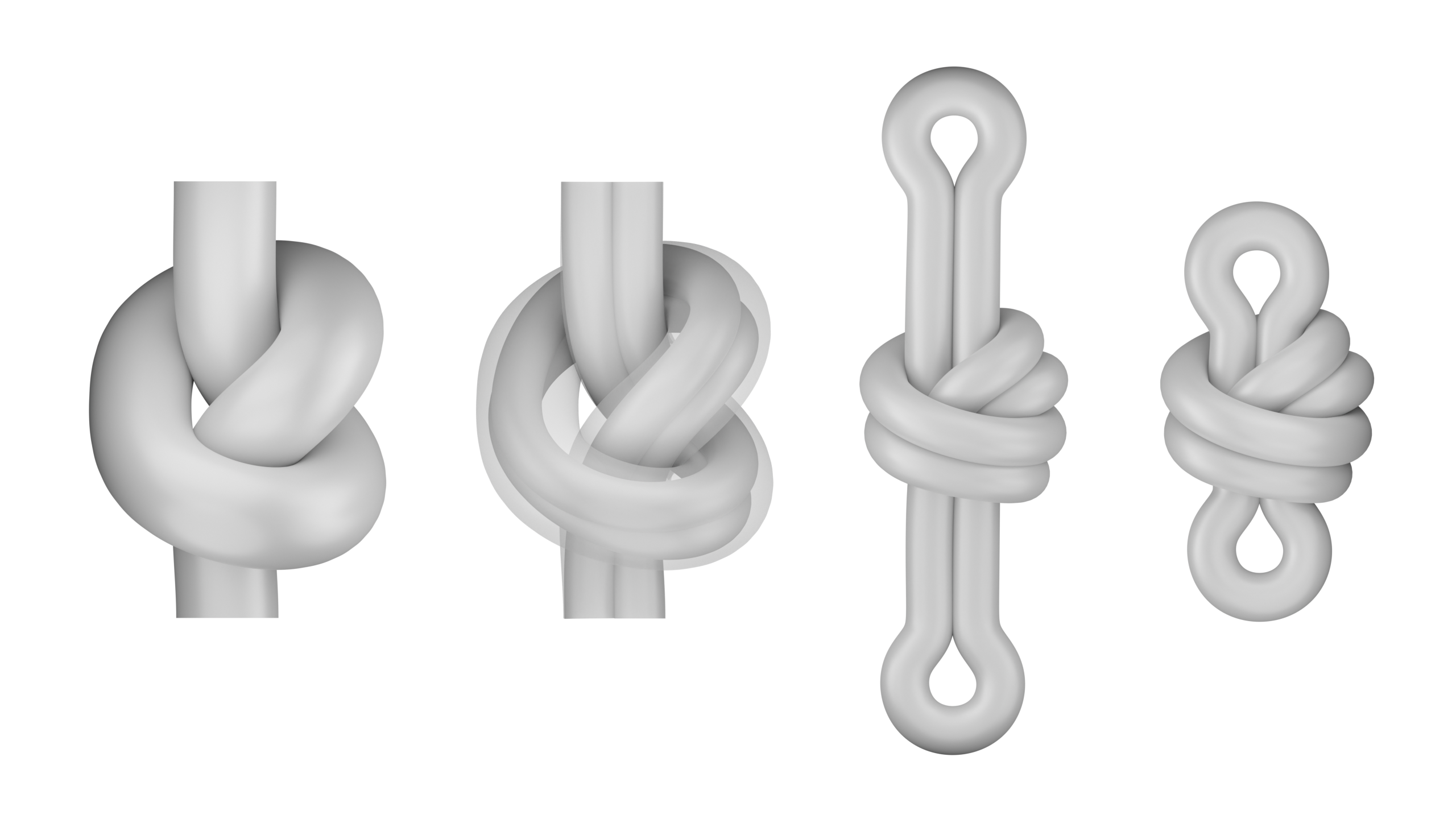}
\end{center}
\caption{From left to right: a nearly tight open trefoil with $\tau = 2$. A pair of embedded, parallel tubes inside the nearly tight open trefoil, each with $\tau = 1$. A nearly tight double overhand knot capped with Dubins curves. A gordian unknot.}

\label{fig:render9}
\end{figure}}

\smallskip
\noindent\textbf{The aperture contour.}
Let \( K_0 \subset \mathbb{R}^3 \) be the double overhand and let \( A_0 \subset K_0 \) be one of the two thick long arcs squeezed by the knot self-wrapping. We define the aperture contour through the triple \( (\alpha_0, D_0, N_0) \), where:

\begin{enumerate}
    \item \( \alpha_0 \subset \partial K_0 \) is a simple closed curve, the contour,
    \item \( D_0 \subset \mathbb{R}^3 \) is a topological disk bounded by \( \alpha_0 \), intersecting \( \partial K_0 \) transversely,
    \item \( N_0 = \left\{ x \in D_0 \,\middle\vert\, \mathrm{reach}(\partial K_0, x) < 2r \right\} \) is the near-contact region, i.e., where the local reach of the tube drops below the diameter.
\end{enumerate}

This encodes a physical bottleneck \( D_0 \) that separates the long arc \( A_0 \) from the rest of the knot, while \( N_0 \) captures geometric obstructions arising from nearby parts of the tube. The condition on reach ensures that \( N_0 \) includes not just close proximity but also curvature-induced constraints. During any isotopy \( \gamma_t \), the triple evolves to \( (\alpha_t, D_t, N_t) \), and persistence of the bottleneck is expressed as:
\[
\inf_{t \in [0,1]} \mathrm{diam}(D_t) > 0,
\quad
\liminf_{t \to t_0} \mathrm{Area}(N_t) > 0.
\]

\smallskip
\noindent\textbf{Cone angle collapse.}
We claim that no admissible isotopy \( \gamma_t \) exists from the double overhand unknot \( K_0 \) to a thickened round circle \( K_1 \). The triple \( (\alpha_t, D_t, N_t) \) defines a physical bottleneck separating a long arc \( A_0 \subset K_0 \) from the remainder of the tube. Let \( p_0 \in A_0 \) be a base point from which this separation is viewed as a spatial cone with aperture contour \( \alpha_0 \) and cone angle:
\[
\theta(0) = \sup_{x, y \in \alpha_0} \angle(x - p_0, y - p_0).
\]

\noindent Under any admissible isotopy \( \gamma_t \), the triple \( (\alpha_t, D_t, N_t) \) evolves continuously or upper semicontinuously. By assumption, the long arc eventually passes through the aperture: there exists a first time \( t_1 \in (0,1) \) at which it intersects \( D_{t_1} \), crossing from one side to the other.

Let \( p_t \in K_t \) be a tip of the thick long arc just before this transition. Define the evolving cone angle:
\[
\theta(t) = \sup_{x,y \in \alpha_t} \angle(x - p_t, y - p_t).
\]
Since the isotopy ends in \( K_1 \), the aperture necessarily vanishes:
\[
\liminf_{t \to t_1} \theta(t) = 0.
\]
Hence, for some \( t < t_1 \), the aperture satisfies \( \theta(t) < \theta(0) \). This collapse leads to a contradiction in two distinct ways:

\begin{enumerate}
    \item {Curvature violation:} The long arc \( A_t \) is continuously squeezed through the narrowing cone into a shorter arc with endpoints within \( D_t \). Then by Lemma~\ref{r1}, its curvature must increase, violating the 1-constrained condition.
    
    \item{Thickness violation:} The persistence of \( N_t \), defined via the local drop in reach, implies that even as \( \theta(t) \to 0 \), the disk remains obstructed by nearby parts of the tube. If \( D_t \) were to contract to permit passage while preserving \( N_t \), the local geometry would force a decrease in tube radius or reach, violating the unit-thickness constraint.
\end{enumerate}

In both cases, the aperture triple exposes an incompatibility between the isotopy and the geometric constraints. Furthermore, since \( \alpha_t \subset \partial K_t \) evolves continuously and continues to separate the long arc, the disk cannot collapse:
\[
\inf_{t \in [0,1]} \mathrm{diam}(D_t) > 0.
\]

Similarly, the obstruction persists under any thickness-preserving isotopy:
\[
\liminf_{t \to t_1} \mathrm{Area}(N_t) > 0,
\]
since the near-contacts in \( N_0 \) arise from self-wrapping and cannot be separated without violating unit thickness. By semicontinuity, this ensures the bottleneck remains nontrivial throughout the isotopy, completing the proof.
\end{proof}

Although the underlying knot types are topologically equivalent, the obstruction arises purely from geometric constraints which prevent an admissible isotopy between them.

\begin{proposition}\label{prop:U1-infinite}
The space \( \mathcal{U}_1 \) of physical unknots with thickness and curvature bounded by one contains infinitely many distinct isotopy classes. That is,
$
\pi_0(\mathcal{U}_1)=\infty.
$
\end{proposition}

\begin{proof}
Let \( K_1 \subset \mathcal{U}_1 \) be a thick double overhand unknot constructed with tube radius \( 1/2 \), as described in Theorem~\ref{main1}. For each \( n \geq 1 \), construct a new knot \( K_n \) by cutting the Dubins caps and stacking \( n \) disjoint vertically aligned copies of \( K_1 \), joined by short vertical tubes and capped at the end with planar Dubins arcs to form a closed \( 1 \)-constrained curve.

Each configuration \( K_n \) lies in \( \mathcal{U}_1 \) and inherits geometric obstructions to isotopy from the original overhand core. By the cone-angle and bottleneck obstruction argument in Theorem~\ref{main1}, no two of these unknots can be connected via a curvature and thickness preserving isotopy. Thus, the knots \( \{K_n\}_{n \geq 1} \) represent infinitely many distinct elements in \( \pi_0(\mathcal{U}_1) \).
\end{proof}


\begin{corollary}
The stratified union of physical unknots of thickness at most one contains infinitely many distinct isotopy classes:
\[
\left| \bigcup_{\tau \leq 1} \pi_0(\mathcal{U}_\tau) \right| = \infty.
\]
\end{corollary}

\begin{proof}
By Proposition~\ref{prop:U1-infinite}, the space \( \mathcal{U}_1 \) contains infinitely many distinct isotopy classes. Since the filtration satisfies
\[
\mathcal{U}_1 \subset \mathcal{U}_\tau \subset \mathcal{U}_0 \quad \text{for all } \tau < 1,
\]
each isotopy class in \( \mathcal{U}_1 \) remains admissible for all smaller values of \( \tau \). Therefore, the union over all \( \tau \leq 1 \) also contains infinitely many distinct classes.
\end{proof}


{\bf{Final Comments.}}
The presented stratified theory reveals a transition in the geometry of physical unknots as thickness varies. For sub maximal thickness $\tau < 2$, the space of unknots fragments into infinitely many isotopy classes, enabled by a rich repertoire of geometric obstructions: aperture bottlenecks that trap long arcs, waist constraints that resist compression, self-contact patterns that prevent unraveling. This intricate stratification emerges from the flexibility of thin tubes, where curvature constraints permit Dubins-type deformations while still maintaining knottedness through prescribed local geometries.

At the critical value \( \tau = 2 \), the behaviour changes dramatically. The tube saturates its curvature-limited normal injectivity radius, eliminating the geometric slack that enabled thin knot constructions. Whereas thin knots exploit controlled buckling and localised pinching, thick knots become globally rigid: maximal self-avoidance leaves no room for aperture collapse or long-arc passage. This transition is not merely quantitative, it marks a qualitative shift from the flexible landscape of thin knots, with its hierarchical obstructions, to the singular rigidity of the thick regime. Understanding unknots in this saturated geometry may require new tools, drawing from Morse theory, critical point analysis, or the discrete geometry of packing constraints.


\section*{Acknowledgments}

I thank Hyam Rubinstein, Joel Hass and Rob Kusner for their insightful suggestions and inspiring discussions throughout the development of this work.


 \end{document}